\documentclass[reqno, 12pt]{amsart}
\usepackage[utf8]{inputenc}
\usepackage{amssymb,amsmath,amsfonts,amsthm,calrsfs,mathtools}
\usepackage{color}
\usepackage[english]{babel}
\usepackage[T1]{fontenc}
\usepackage{latexsym}
\usepackage{enumitem}
\usepackage[a4paper,top=3cm,bottom=2cm,left=3cm,right=3cm,marginparwidth=1.75cm]{geometry}
\usepackage[colorinlistoftodos]{todonotes}
\usepackage[colorlinks=true, allcolors=blue]{hyperref}
\newtheorem{teor}{Theorem}[section]
\newtheorem{thmx}{Theorem}

\newtheorem{lemma}[teor]{Lemma}
\newtheorem{prop}[teor]{Proposition}
\newtheorem{defi}[teor]{Definition}
\newtheorem{obs}[teor]{Remark}

\providecommand{\keywords}[1]

\title{Hyperbolicity of maximal entropy measures for certain maps isotopic to Anosov diffeomorphisms}
\author{Carlos F. \'{A}lvarez}
\address{Instituto de Matemática, Estatística e Computação Científica, IMECC-UNICAMP, Campinas-SP, Brazil}
\curraddr{Área de Ciencias Básicas Exactas, Grupo de Investigación Deartica, Universidad del Sinú Seccional Cartagena, Av. El Bosque, Trasnversal 54 Nº 30-72, Cartagena de Indias 130001, Colombia}
\email{carlosfalvarez@unisinu.edu.co, fabmath92@gmail.com}
\thanks{This work was partially supported by CAPES-Brazil grant \#88882.329056/2019-01.}

\makeatletter

\begin{document}
\maketitle
\markright{HYPERBOLICITY OF MAXIMAL ENTROPY MEASURES}
\begin{abstract}
We prove the hyperbolicity of ergodic maximal entropy measures for a class of partially hyperbolic diffeomorphisms of $\mathbb{T}^{d}$, which have a compact two-dimensional center foliation. 
\end{abstract}
\section{Introduction} 
The \textit{metric entropy} describes the complexity level of a dynamical system with respect to an invariant probability measure, and the \textit{topological entropy} describes the complexity level of the whole system. If $f$ is a continuous transformation over a compact metric space, then the variational principle \cite[Theorem~8.6]{W} asserts that its topological entropy $h_{top}(f)$ coincides with the supremum of the entropies $h_{\mu}(f)$ with respect to all the invariant probability measures. An invariant probability measure such that $h_{\mu}(f)=h_{top}(f)$ is called a \textit{maximal entropy measure} for $f$. Such measures describe the complexity level of the whole system, which makes them an interesting topic of study. 
\vspace{0.1cm}

An important problem of smooth ergodic theory is to determine existence and uniqueness of maximal entropy measures for diffeomorphisms over compact manifolds. This problem has been solved by Bowen for uniformly hyperbolic diffeomorphisms \cite{B} and by Buzzi, Crovisier and Sarig for $C^{\infty}$ surface diffeomorphisms with positive topological entropy \cite{BCS}. 
\vspace{0.1cm}

In the setting of partially hyperbolic diffeomorphisms with one-dimensional center bundles, the existence of maximal entropy measures has been proved in \cite{DF} and there are some examples the uniqueness of maximal entropy measures where has been proved, for instance, for systems isotopic to an Anosov diffeomorphism \cite{BFSV, FPS, U}. For certain $C^{1+\alpha}$ partially hyperbolic diffeomorphisms over 3-manifolds having compact one-dimensional central leaves, a dichotomy was proved about the number of maximal entropy measures \cite{RRTU}.
\vspace{0.1cm}

The \textit{Lyapunov exponents} are real numbers, which measure the exponential growth of the derivative of dynamical systems (see \ref{lex} for a precise definition). An interesting question is to know under what conditions an ergodic maximal entropy measure is \textit{hyperbolic}, that is, it has no zero Lyapunov exponents and there exist Lyapunov exponents with different signs. 
\vspace{0.1cm}

In this work we study maximal entropy measures for partially hyperbolic systems with two-dimensional center bundles, and we propose a scenario in which every ergodic maximal entropy measure is hyperbolic.
\vspace{0.2cm}

Let $M$ is a compact Riemannian manifold. We start introducing the definition of partial hyperbolicity:
\begin{defi}
A diffeomorphism $f :M \rightarrow M$ is called \textbf{partially hyperbolic} (pointwise) if the tangent bundle admits
a continuous $Df$-invariant splitting $TM=E^{s}\oplus E^{c} \oplus E^{u}$ such that there exists $N>0$ and $\lambda>1$ satisfying the following conditions:
\begin{enumerate}
\item $\lambda\|Df_{x}^{N}v^{s}\|<\|Df_{x}^{N}v^{c}\|<\lambda^{-1}\|Df_{x}^{N}v^{u}\|$, 
\item $\|Df_{x}f^{N}v^{s}\|<\lambda^{-1}<\lambda<\|Df_{x}^{N}v^{u}\|,$ 
\end{enumerate}
for every $x\in M$ and unit vectors $v^{\sigma}\in E^{\sigma}(x) (\sigma=s,c,u).$
\end{defi}

\begin{obs}
For partially hyperbolic diffeomorphisms, it is a well-known fact that there are foliations $\mathcal{F}^{\sigma}$ tangent to the subbundles $E^{\sigma}$ for $\sigma= s,u$ \cite{HPS}. A partially hyperbolic diffeomorphism is called \textbf{dynamically coherent} if there exist invariant foliations $\mathcal{F}^{c \sigma}$ tangent to $E^{c \sigma}=E^{c}\oplus E^{\sigma}$ for $\sigma=s,u.$
\end{obs}

\begin{defi}
A $C^1$-diffeomorphism $f:\mathbb{T}^d \to \mathbb{T}^d$ is called derived from Anosov if it is isotopic to its action in the homology $A:H^1(\mathbb{T}^d)\to H^1(\mathbb{T}^d)$, which is a linear Anosov automorphism with three invariant subbundles $T(\mathbb{T}^d)= E_{A}^u \oplus E^{c}_{A} \oplus E^{s}_{A}.$  Recall that an isotopy is a homotopy which is a homemorphism. 
\end{defi}

The set of all partially hyperbolic diffeomorphisms of $\mathbb{T}^{d}$ is denoted by $\mathsf{PH}(\mathbb{T}^d)$. Let $A:\mathbb{T}^d\to \mathbb{T}^d$ be a linear Anosov with dominated splitting of the form $E_{A}^{ss}\oplus E_{A}^{ws}\oplus E_{A}^{wu}\oplus E_{A}^{uu}$, with $E^{c}_{A}= E_{A}^{ws}\oplus E_{A}^{wu}$ and $\dim E_{A}^{ws}=E_{A}^{wu}=1$. We consider the set of partially hyperbolic diffeomorphisms isotopic to $A$, all of them having the same dimension of stable and unstable bundle, that is, 
$$\mathsf{PH}_{A,s,u}(\mathbb{T}^{d})=\{f\in \mathsf{PH}(\mathbb{T}^d): f\sim A, \ \dim E_{f}^{s}=\dim E^{ss}_{A}, \ \dim E_{f}^{u}=\dim E^{uu}_{A}\}.$$

Here $f\sim A$ denotes an isotopy between $f$ and $A$. To simplify notation we will denote $\mathsf{PH}_{A,s,u}(\mathbb{T}^{d})$ as $\mathsf{PH}_{A}(\mathbb{T}^d)$, where the dimension of the bundles is implicitly understood. Now we consider $\mathsf{PH}^{0}_{A}(\mathbb{T}^{d})$ to be the connected component of $\mathsf{PH}_{A}(\mathbb{T}^{d})$ containing $A.$
Then, in \cite[Corollary~C]{FPS} it is proved that:

\begin{teor}[Fisher, Potrie, Sambarino]\label{mfps}
If $f\in \mathsf{PH}^{0}_{A}(\mathbb{T}^{d})$ and $\dim E^{c}_{f}=1$, then there is a unique maximal entropy measure which has entropy equal to the linear part.
\end{teor}

Recall that for every continuous map $f:\mathbb{T}^d\to \mathbb{T}^d$ there exists a unique linear map $A_{f}:\mathbb{R}^d\to \mathbb{R}^d$ such that the lift of $f$ to the universal covering map $\tilde{f}$ is homotopic to $A_{f}.$ We call $A_f$ the \textit{linear part} of $f$.
As consequence of \cite[Theorem~A]{R} the unique maximal entropy measure given by Theorem \ref{mfps} is a hyperbolic measure.
\vspace{0.2cm}

Since topological entropy plays such a central role, we mention the relationship between $h_{top}(f)$ and $h_{top}(A)$, in general, when $A$ is a factor of $f$ one just has $h_{top}(f)\geq h_{top}(A)$. 

Now, let $A_{c}:\mathbb{T}^{d}/\mathcal{F}^{c}_{A}\to \mathbb{T}^{d}/\mathcal{F}^{c}_{A}$ be the corresponding factor to the linear Anosov $A$. Here, $\mathbb{T}^{d}/\mathcal{F}^{c}_{A}$ denote the space of central leaves of $A$. In the context presented above, we obtain the following results:

\begin{thmx}\label{main:hyperbolicity}
Suppose $A$ admits a foliation by tori $\mathbb{T}^2$ tangent to $E^c_{A}$. Let $f\in \mathsf{PH}^{0}_{A}(\mathbb{T}^{d})\cap \mathrm{Diff}^{2}(\mathbb{T}^{d})$ and let $k_{0}:=h_{top}(A_{c})$. If $\mu$ is an ergodic measure such that $h_{\mu}(f)>k_{0}$, then $\mu$ is a hyperbolic measure with both positive and negative Lyapunov central exponents. In particular any maximal entropy measure is hyperbolic, provided that it exists.
\end{thmx}


The class of systems introduced in \cite[Section~5]{BFSV} satisfies the hypothesis of the previous theorem. Such examples are partially hyperbolic systems produced by an isotopy from an Anosov system. Moreover, for more examples of systems that satisfy this type of hypotheses, see \cite{FPS} and references therein.

\begin{obs}
Useful references are {\rm\cite{CLPV, R}}, in which we can find examples of partially hyperbolic diffeomorphisms $f:\mathbb{T}^{4}\rightarrow \mathbb{T}^{4}$ with two-dimensional center bundle admitting high entropy measures, that is, measures $\mu$ such that $h_{\mu}(f)\geq h_{top}(A)$. The proof of the hyperbolicity for maximal entropy measures is the same for high entropy measures.
\end{obs}




The remainder of the article is organized as follows. In the next section, we discuss some necessary preliminaries about entropy, Pesin theory and partially hyperbolic dynamics. In Section \ref{s4} 
we present the proof of our result.

\section{Definitions and preliminaries}
In this section, we introduce definitions and results about complexity of dynamical systems, Pesin theory and partially hyperbolic diffeomorphisms. 

\subsection{Entropy and Lyapunov exponents}\label{led}
We give an alternative definition of metric entropy as defined by Brin and Katok, with respect to a $f$-invariant ergodic probability measure. 
\vspace{0.1cm} 

Let $f:M\to M$ be a continuous map over a compact metric space and $\mu$ be an ergodic $f$-invariant probability measure. For $\delta\in (0,1)$, $n\in \mathbb{N}$ and $\epsilon>0$, a finite set $E\subset M$ is called an $(n,\epsilon,\delta)$-\textit{covering} if the union of the all $\epsilon$-balls, $B_{n}(x,\epsilon)=\{y\in M: d(f^{i}(x),f^{i}(y))<\epsilon\}$, centered at points $x\in E$ has $\mu$-measure greater than $1-\delta$. The \textit{metric entropy} is defined by $$h_{\mu}(f)=\displaystyle\lim_{\epsilon \rightarrow 0}\displaystyle\limsup_{n\rightarrow \infty}\frac{1}{n}\log \min\{\# E: E\subseteq M \ {\rm is  \ a \ }(n,\epsilon,\delta)-{\rm covering \ set} \}.$$

A set $E\subseteq M$ is said to be $(n,\epsilon)$-\textit{separated}, if for every $x, y \in E, x \neq y$, there exists $i\in\{0,\ldots, n-1\}$ such that $d(f^{i}x, f^{i}y)\geq \epsilon$. The \textit{topological entropy} is defined by $$h_{top}(f)=\displaystyle\lim_{\epsilon \rightarrow 0}\displaystyle\limsup_{n\rightarrow \infty}\frac{1}{n}\log \sup\{\# E: E\subseteq M \ is \ (n, \epsilon){\rm-separated} \}.$$

The variational principle states that $$\sup \{h_{\mu}(f):\mu\in \mathcal{M}(f)\}=\sup \{h_{\mu}(f):\mu\in \mathcal{M}_{e}(f)\}=h_{top}(f),$$ where $\mathcal{M}(f)$ denotes the set of $f$-invariant Borel probability measures on $M$ and $\mathcal{M}_{e}(f)$ denotes the set of ergodic measures on $M$ ($\mathcal{M}_{e}(f)\subset \mathcal{M}(f)$). 
\vspace{0.1cm}

A probability measure $\mu\in \mathcal{M}(f)$ such that $h_{\mu}(f)=h_{top}(f)$ is called a \textit{maximal entropy measure}. It is well-known that almost every ergodic component of a maximal entropy measure is an ergodic maximal entropy measure \cite{W}. For simplicity of the exposition, we are supposing that maximal entropy measures are ergodic.
\vspace{0.1cm}

For $r\geq 1$, let $f\in \mathrm{Diff}^{r}(M)$ be a $C^{r}$ diffeomorphism over a compact Riemannian manifold $M$ of dimension $d$. Given a vector $v\in T_{x}M$, the \textit{Lyapunov exponent} of $f$ at $x$ is the exponential growth rate of $Df$ along $v$, that is $$\lambda(x,v)=\displaystyle\lim_{n\rightarrow \pm \infty}\frac{1}{n}\log \|Df^{n}_{x}(v)\|$$ in case this amount is well defined.\label{lex} For every ergodic $f$-invariant probability measure $\mu$, from Oseledets' theorem (see \cite[Theorem~10.1]{M}) 
it follows that there exist real numbers $\lambda_{1}\leq \lambda_{2}\leq \cdots \leq \lambda_{d}$, called \textit{Lyapunov exponents} and a decomposition of the tangent bundle $T_{x}M=E_{1}(x)\oplus E_{2}(x)\oplus \cdots \oplus E_{d}(x)$ for $\mu$-a.e $x\in M$ such that for every $1\leq i\leq d$ and every $v\in E_{i}(x)-\{0\}$ we have $$\lambda(x,v)=\lambda_{i}.$$ 

The numbers $\lambda_{i}$ denote the Lyapunov exponents of the measure $\mu$ and $x$ is called a \textit{regular point}.

\begin{defi}\label{defhm}
An ergodic $f$-invariant probability measure $\mu$ is called \textbf{hyperbolic} if all its Lyapunov exponents are non-zero and there exist Lyapunov exponents with different signs. We say that a hyperbolic measure has \textit{saddle type} if $\lambda_{1}<0<\lambda_{d}.$ We denote $\mathcal{M}_{h}(f)$ the set of ergodic hyperbolic measures of saddle type.
\end{defi}

Given a compact $\Lambda\subset M$, invariant under $f\in \mathrm{Diff}^{r}(M), \ r\geq 1$, we say that $\Lambda$ admits a \textit{dominated splitting} if there exists a $Df$-invariant decomposition of the tangent bundle $T_{\Lambda}M=E\oplus F$, and constants $C > 0$, $\lambda\in (0,1)$ such that for every $x\in \Lambda$ and $n>0$ one has $$\|Df^{n}_{x}|_{E}\|\|Df_{f^{n}(x)}^{- n}|_{F}\|\leq C\lambda^{n}.$$

A \textit{hyperbolic set} for $f$ is a compact $f$-invariant set $\Lambda\subset M$ with a decomposition $T_{x}M=E^{s}(x)\oplus E^{u}(x)$ for all $x\in \Lambda$ such that for some $C>0$ and $\lambda\in (0,1),$ for all $x\in \Lambda$, $n\geq 0,$ $v^{s}\in E^{s}(x)$ and $v^{u}\in E^{u}(x),$ we have $$\|Df_{x}^{n}(v^{s})\|\leq C\lambda^{n}\|v^{s}\| \ {\rm{and}}  \ \|Df_{x}^{-n}(v^{u})\|\leq C\lambda^{n}\|v^{u}\|.$$ 
When $\Lambda=M$, $f$ is called \textbf{Anosov diffeomorphism}. By \cite{HPS}, if $\Lambda$ is a hyperbolic set, then for every $x\in \Lambda$ we define the following sets:
\begin{align*}
 W^{u}(x)&= \{y\in M: d(f^{-n}(x),f^{-n}(y))\xrightarrow{n\rightarrow\infty} 0\}, \\ 
 W^{s}(x)&= \{y\in M: d(f^{n}(x),f^{n}(y))\xrightarrow{n\rightarrow\infty} 0\}.
\end{align*}
These are $C^{r}$ injectively immersed sub-manifolds, called unstable and stable manifold, respectively. 





\subsection{Pesin Theory}
Let us present some results about Pesin theory based on the references \cite{BP, K1, P}. \vspace{0.1cm}

Let $\mu$ be a $f$-invariant probability measure, from Oseledets' theorem there exists a regular set $\mathcal{R}$ with $\mu(\mathcal{R})=1$. For all $x\in \mathcal{R}$ we have $$T_{x}M=\bigoplus_{\lambda_{i}(x)<0}E_{i}(x)\oplus E^{0}(x)\bigoplus_{\lambda_{i}(x)>0}E_{i}(x)$$ where $E^{0}(x)$ is the subspace generated by the vectors having zero Lyapunov exponents. For fixed $\epsilon>0$ and each $l\in \mathbb{N}$ we have that $$\mathcal{R}=\displaystyle\bigcup_{l=1}^{\infty}\mathcal{R}_{\epsilon,l},$$ where $\mathcal{R}_{\epsilon,l}$ are the \textit{Pesin blocks}. These sets are nested and compact. 

\begin{defi}
Let $\mu$ be an ergodic hyperbolic measure. We define the \textbf{stable index} or $s$-\textbf{index} (\textbf{unstable index} or $u$-\textbf{index}) of $\mu$ as the number of negative (positive) Lyapunov exponents. Here the exponents are counted with multiplicity.
\end{defi}

\begin{defi}
For $f\in \mathrm{Diff}^{r}(M)$ with $r>1$, the \textbf{stable Pesin manifold} of the point $x\in \mathcal{R}$ is $$W^{s}_{P}(x)=\{y\in M: \displaystyle\limsup_{n\to \infty}\frac{1}{n}\log d(f^{n}(x),f^{n}(y))<0\}.$$
Similarly one defines the \textbf{unstable Pesin manifold} as $$W^{u}_{P}(x)=\{y\in M: \displaystyle\limsup_{n\to \infty}\frac{1}{n}\log d(f^{-n}(x),f^{-n}(y))<0\}.$$
\end{defi}

\begin{obs}\label{defmho}
Stable and unstable Pesin manifolds of points in $\mathcal{R}$ are injectively immersed $C^r$-submanifolds. The usual Pesin theory requires a $C^{1+\alpha}$ regularity $(\alpha>0)$ of the dynamics \cite{BP}.
\end{obs}

Let us call $W^{s}_{loc}(x)$ the connected component of $W^{s}_{P}(x)\cap B(x,r)$ containing $x$, where $B(x,r)$ denotes the Riemannian ball of center $x$ and radius $r=r(\epsilon,l)>0$, which is sufficiently small but fixed. 

\begin{teor}[Stable Pesin Manifold Theorem \cite{P}]\label{vep}
Let $\alpha>1$ and $f\in \mathrm{Diff}^{\alpha}(M)$ be a diffeomorphism preserving a smooth measure $m$. Then, for each $l\in \mathbb{N}$ and small $\epsilon>0$, if $x\in \mathcal{R}_{\epsilon,l}$:
\begin{enumerate}
\item $W^{s}_{loc}(x)$ is a disk such that $T_{x}W^{s}_{loc}(x)=\bigoplus_{\lambda_{i}(x)>0}E_{i}(x)$;
\item $x\mapsto W^{s}_{loc}(x)$ is continuous over $\mathcal{R}_{\epsilon,l}$ in the $C^{1}$ topology.
\end{enumerate}
\end{teor}

In particular, the dimension of the disk $W^{s}_{loc}(x)$ equals the number of negative Lyapunov exponents of $x$ with respect to $m$. An analogous statement holds for the unstable Pesin manifold.

\begin{obs}\label{op}
The global stable manifold states that for any $l_{k}\to +\infty$ satisfying $f^{l_{k}}(x)\in R_{\epsilon,l},$ one has $W_{P}^{s}(x)=\bigcup_{k}f^{-l_{k}}(W^{s}_{loc}(f^{l_{k}}(x))).$ Similarly, any $l_{k}\to +\infty$ satisfying $f^{-l_{k}}(x)\in R_{\epsilon,l},$ one has $W_{P}^{u}(x)=\bigcup_{k}f^{l_{k}}(W^{u}_{loc}(f^{-l_{k}}(x))).$ This property follows from the properties of the local stable manifold {\rm (Theorem \ref{vep})}. 
\end{obs}

\subsection{Partially hyperbolic diffeomorphisms isotopic to Anosov}
Let us consider $f:\mathbb{T}^{d}\rightarrow \mathbb{T}^{d}$ a $C^{2}$ partially hyperbolic diffeomorphism. Suppose $f$ is isotopic to $A:\mathbb{T}^{d}\rightarrow\mathbb{T}^{d}$ along a path of partially hyperbolic diffeomorphisms, where $A:\mathbb{T}^{d}\rightarrow\mathbb{T}^{d}$ is a linear Anosov automorphism with a dominated splitting and foliation by tori $\mathbb{T}^{2}$ tangent to $E^{c}_{A}$. Thus, $f\in \mathsf{PH}_{A}^{0}(\mathbb{T}^{d})\cap \mathrm{Diff}^{2}(\mathbb{T}^{d})$.
\vspace{0.1cm}

By a classical Franks-Manning result \cite{F,MA} there exists a continuous surjection $H:\mathbb{T}^{d}\rightarrow \mathbb{T}^{d}$ homotopic to the identity such that  \begin{equation}\label{2.1}
 A\circ H=H\circ f.  
\end{equation}
Moreover, its lift $\widetilde{H}:\mathbb{R}^{d}\to  \mathbb{R}^{d}$ is a proper function that semiconjugates $\widetilde{f}:\mathbb{R}^{d}\to  \mathbb{R}^{d}$ with $\widetilde{A}:\mathbb{R}^{d}\to  \mathbb{R}^{d}$, and for some constant $K>0$, we have $$\|\widetilde{H}-Id\|_{C^{0}}\leq K.$$ 

Using the Franks-Manning semiconjugacy and other properties, in \cite[Theorem~A]{FPS} it is proved that:
\begin{teor}[Fisher-Potrie-Sambarino]\label{2.2}
If $f\in\mathsf{PH}_{A}^{0}(\mathbb{T}^{d})\cap \mathrm{Diff}^{2}(\mathbb{T}^{d})$, then $f$ is dynamically coherent. Moreover, if $A$ admits a center foliation by tori $\mathbb{T}^{2}$, then $f$ admits a center foliation with all central leaves compact.
\end{teor}

\subsection{Partially hyperbolic diffeomorphisms with all center leaves compact}
We denote by $\mathcal{F}^{c}_{loc}(x)$ the connected component of $\mathcal{F}^{c}(x)\cap B(x,r)$ containing $x$. For any $f:M\rightarrow M$ dynamically coherent partially hyperbolic diffeomorphism and any two points $x,y$ with $y\in \mathcal{F}^{u}(x)$, there exists a neighborhood $U_{x}$ of $x$ in $\mathcal{F}^{c}(x)$ and a homeomorphism $H^{u}_{x,y}:U_{x}\rightarrow \mathcal{F}^{c}(y)$ such that $H_{x,y}^{u}(x)=y$ and $H^{u}_{x,y}(z)\in \mathcal{F}^{u}(z)\cap \mathcal{F}^{c}_{loc}(y).$ The homeomorphisms $H^{u}_{x,y}$ are called \textit{local unstable holonomies} for $y\in \mathcal{F}^{u}(x)$. Similarly, one may define \textit{local stable holonomies} $H^{s}_{x,y}$ for $y\in \mathcal{F}^{s}(x).$
\vspace{0.1cm}

We say that $f$ admits global unstable holonomy if for any $y\in\mathcal{F}^{u}(x)$ the holonomy is defined globally $H_{x,y}^{u}:\mathcal{F}^{c}(x)\rightarrow \mathcal{F}^{c}(y)$. Similarly, we define the notion of global stable holonomy, and $f$ admits \textit{global holonomies} when it admits global stable and unstable holonomies.
\vspace{0.1cm}

The following results can be found in \cite{TY}. Let $f:M\to M$ be a $C^2$ partially hyperbolic diffeomorphism over a compact manifold satisfying the following conditions:
\begin{enumerate}[label=H.\arabic*]
\item \label{H1} $f$ is dynamically coherent with all center leaves compact;
\item \label{H2} $f$ admits global holonomies;
\item \label{H3} $f_{c}$ is a transitive topological Anosov homeomorphism (see \cite[Chapter~4]{M}), where $f_{c}:M/\mathcal{F}^{c}\rightarrow M/\mathcal{F}^{c}$ is the induced dynamics satisfying $f_{c}\circ \pi=\pi\circ f$ and $\pi:M\rightarrow M/\mathcal{F}^{c}$ is the natural projection to the space of central leaves.
\end{enumerate}

\begin{obs}
A transitive topological Anosov homeomorphism shares many similar properties with Anosov diffeomorphisms. For example, it has a pair of topological foliations which play the same role of stable/unstable foliations and it has a unique maximal entropy measure. 
\end{obs}

\begin{defi}
For every $f$-invariant probability $\mu$, we say a measurable partition $\xi$ is \textbf{$\mu$-adapted} (subordinated) to the foliation $\mathcal{F}$ if the following conditions are satisfied:
\begin{enumerate}
\item There is $r_{0}>0$ such that $\xi(x)\subset B^{\mathcal{F}}_{r_{0}}(x)$ for $\mu$ almost every $x$ where $B^{\mathcal{F}}_{r_{0}}(x)$ is the ball inside of the leaf $\mathcal{F}(x);$
\item $\xi(x)$ contains an open neighborhood of $x$ inside $\mathcal{F}(x);$
\item $\xi$ is increasing; that is, for $\mu$ almost every $x,\ \xi(x)\subset f(\xi(f^{-1}(x))).$ 
\end{enumerate}
\end{defi}

\begin{defi}
For every $f$-invariant probability $\mu$, the \textbf{partial entropy} of $f$ along the expanding foliation $\mathcal{F}^{u}$ is defined by \begin{equation*}
    h_{\mu}(f,\mathcal{F}^{u})=H_{\mu}(f^{-1}\xi^{u}|\xi^{u})=\displaystyle\int_{M}-\log \mu_{z}^{u}(f^{-1}\xi^{u}(z))d\mu(z),
\end{equation*}
where $\xi^{u}$ is a partition $\mu$-adapted to the foliation $\mathcal{F}^{u}$ and $\{\mu_{z}^{u}\}$ is the conditional probabilities along leaves of  $\mathcal{F}^{u}$.
\end{defi}

\begin{obs}
The previous definition of partial entropy is different from the Ledrappier-Young unstable entropy, because we may have positive exponents in the central direction. 
\end{obs}

\begin{teor}[Tahzibi-Yang]\label{2.4}
Let $f$ be a $C^{2}$ partially hyperbolic diffeomorphism satisfying {\rm\ref{H1}, \ref{H2}} and {\rm\ref{H3}}. Suppose $\mu$ to be an $f$-invariant probability measure. Then, $h_{\mu}(f,\mathcal{F}^{u})\leq h_{\pi_{\ast}\mu}(f_{c})$. 
\end{teor}

\section{Proof of Theorem \ref{main:hyperbolicity}}\label{s4}
We start this section with an important remark. 

\begin{obs}
For $f\in \mathsf{PH}(\mathbb{T}^d)$ there exist $f$-invariant foliations $\mathcal{F}^{ss}$ and $\mathcal{F}^{uu}$ tangent to $E^{ss}_{f}$ and $E^{uu}_{f}$
respectively that we call the \textit{strong foliations}. Moreover, $\mathcal{F}^{cs}$ and $\mathcal{F}^{cu}$ are tangent to $E_{f}^{c}\oplus E_{f}^{ss}$ and $E_{f}^{c}\oplus E_{f}^{uu}$ respectively, when $f$ is dynamically coherent. 
\end{obs}


\begin{lemma}\label{gh}
Every $f:\mathbb{T}^{d}\rightarrow\mathbb{T}^{d}$ under conditions of Theorem \rm{\ref{main:hyperbolicity}} admits global holonomies, that is, for every $x, y\in \mathbb{T}^{d}$ with $y\in \mathcal{F}^{u}(x)$ and every $z\in \mathcal{F}^{c}(x)$ there is a unique $w$ such that $w \in \mathcal{F}^{u}(z)\cap \mathcal{F}^{c}(y).$ 
\end{lemma}
\begin{proof}
Let $z\in \mathcal{F}^{c}(x)$ and $H$ be the semiconjugacy between $f$ and $A$. Theorem \rm{B} of \cite{FPS} implies that $H(z)\in H(\mathcal{F}^{c}(x))=\mathcal{F}_{A}^{c}(H(x)).$ It follows from general results of \cite{HPS} that $A$ admits global holonomies, then there is a unique $\tilde{w}\in \mathcal{F}_{A}^{uu}(H(z))\cap \mathcal{F}^{c}_{A}(H(y))$ with $H(y)\in \mathcal{F}^{uu}_{A}(H(x))$. From injectivity of $H$ in the strong unstable leaves (see \cite[Theorem~4.1]{FPS}), we have that there exists unique $w = H^{-1}(\tilde{w})$ such that
$$w\in H^{-1}[\mathcal{F}_{A}^{uu}(H(z))\cap \mathcal{F}_{A}^{c}(H(y))]\supseteq \mathcal{F}^{u}(z)\cap \mathcal{F}^{c}(y).$$ Therefore, for every $x, y\in \mathbb{T}^{d}$ with $y\in \mathcal{F}^{u}(x)$ and every $z\in \mathcal{F}^{c}(x)$ we have that $\mathcal{F}^{u}(z)\cap \mathcal{F}^{c}(y)$ is a singleton. 
\end{proof}

The following proposition is contained in the proof of Corollary $2.1$ of \cite{TY}, where some results of Ledrappier-Young \cite{LYI}, \cite{LYII} are used. Let us consider the topological quotient $\mathbb{T}^{d}/ \mathcal{F}^{c}$ and the  projection $ \pi: \mathbb {T}^{d} \rightarrow \mathbb{T}^{d} / \mathcal{F}^{c}$ such that the transitive topological Anosov homeomorphism $f_{c}: \mathbb{T}^{d} / \mathcal{F}^{c} \rightarrow \mathbb{T}^{d} / \mathcal {F}^{c}$ satisfies $ \pi\circ f = f_{c}\circ\pi$. 

\begin{prop}\label{2.5}
Let $f$ be as in Theorem \ref{main:hyperbolicity}. If $\mu$ is an ergodic probability with all the central Lyapunov exponents non-positive almost everywhere, then $h_{\mu}(f)=h_{\pi_{\ast}\mu}(f_{c}).$
\end{prop}
\begin{proof}
From Lemma \ref{gh} and previous comments every $f$ under conditions of Theorem \ref{main:hyperbolicity} satisfies the hypothesis of Tazhibi-Yang's Theorem \ref{2.4}. 

The authors in \cite{LYII} define the notion of entropy $h_{i}=h_{\mu}(f, W_{P}^{i})$ along the $i$-th unstable manifold $W_{P}^{i}$ for $1\leq i \leq u$. Here $$W_{P}^{i}(x)=\{y\in \mathbb{T}^{d} : \displaystyle\limsup_{n\rightarrow \infty}\frac{1}{n}\log d(f^{-n}(x), f^{-n}(y))\leq -\lambda_{i}\}$$ and $\lambda_{1}> \lambda_{2}>\cdots > \lambda_{u}$ are the positive Lyapunov exponents of $(f,\mu)$. As $f$ is partially hyperbolic with non-positive central Lyapunov exponent, it follows that $W_{P}^{u}$ coincides with the unstable foliation $\mathcal{F}^{u}.$ By Corollary $7.2.2$ of Ledrappier-Young \cite{LYII} we have that $h_{\mu}(f)=h_{u}$. Again using \cite{LYII} we see that $h_{u}=h_{\mu}(f,\mathcal{F}^{u})$, and consequently $$h_{\mu}(f)=h_{u}=h_{\mu}(f,\mathcal{F}^{u}).$$
On the other hand, Theorem \ref{2.4} implies that $h_{\mu}(f)=h_{\mu}(f,\mathcal{F}^{u})\leq h_{\pi_{\ast}\mu}(f_{c})$, and as $f_{c}$ is factor of $f$, we have that $h_{\mu}(f)=h_{\pi_{\ast}\mu}(f_{c}).$
\end{proof}

\begin{proof}[Proof of Theorem \rm{\ref{main:hyperbolicity}}] Let $\lambda_{1}^{c}, \lambda_{2}^{c}$ be the central Lyapunov exponents of $(f,\mu$) and define $k_{0}=h_{top}(A_{c}),$ where $A_{c}:\mathbb{T}^{d}/\mathcal{F}_{A}^{c}\to \mathbb{T}^{d}/\mathcal{F}_{A}^{c}$. By Ruelle's inequality it follows that
\begin{equation}\label{eq1}
h_{top}(A_{c})<h_{\mu}(f)\leq\max\{\lambda_{1}^{c},0\}+\max\{\lambda_{2}^{c},0\}+\sum \lambda_{i,f}^{+},
\end{equation}
where $\lambda_{i,f}^{+}$ are the positive (unstable) Lyapunov exponents of $f$. By \cite[Theorem~B]{FPS} we can define a homeomorphism between central leaves and obtain that $h_{top}(f_{c})=h_{top}(A_{c}).$ Suppose  $\max\{\lambda_{1}^{c},0\}+\max\{\lambda_{2}^{c},0\}=0,$ then the central Lyapunov exponents are non-positive. From Proposition \ref{2.5} follows that 
\begin{equation*}
h_{\mu}(f)=h_{\mu}(f,\mathcal{F}^{u})=h_{\pi_{\ast}\mu}(f_{c})\leq h_{top}(f_{c})=h_{top}(A_{c}),    
\end{equation*}
this is a contradiction with (\ref{eq1}). Then, $\max\{\lambda_{1}^{c},0\}+\max\{\lambda_{2}^{c},0\}>0.$ Analogously, as  $h_{\mu}(f^{-1})=h_{\mu}(f)$ we have that $\max\{-\lambda_{1}^{c},0\}+\max\{-\lambda_{2}^{c},0\}>0$. Therefore, $\mu$ is a hyperbolic measure with both positive and negative central Lyapunov exponents. 
\end{proof}

\section*{Acknowledgments}
The results of this work are part of author's PhD thesis defended at the IMECC-UNICAMP. The author thanks Rafael Potrie, Régis Var\~{a}o and referees for the useful comments for improve this work. The author also thanks Universidad del Sinú Seccional Cartagena for the financial support for end this project.

\end{document}